\documentclass[10pt]{article}
\usepackage[utf8]{inputenc}

\usepackage{amssymb,amsthm,amsmath}
\usepackage{mathrsfs}
\usepackage{enumerate}
\usepackage{graphicx,xcolor}
\usepackage{setspace}
\usepackage[hidelinks]{hyperref}

\newcommand{\dd}{\mathrm{d}}
\newcommand{\E}{\mathbb{E}}

\newcommand{\1}{\textbf{1}}
\newcommand{\R}{\mathbb{R}}
\newcommand{\C}{\mathbb{C}}

\newcommand{\p}[1]{\mathbb{P}\left( #1 \right)}
\newcommand{\scal}[2]{\left\langle #1, #2 \right\rangle}

\newtheorem{theorem}{Theorem}
\newtheorem{lemma}[theorem]{Lemma}

\theoremstyle{remark}
\newtheorem{remark}[theorem]{Remark}

\theoremstyle{definition}

\title{\vspace{-3em}
Khinchin inequalities for uniforms on spheres with a deficit
}

\author{Jacek Jakimiuk$^\dagger$, \ Colin Tang$^*$, \
Tomasz Tkocz\footnote{Email: ttkocz@math.cmu.edu. Research supported in part by NSF grant DMS-2246484.}
}

\date{\begin{normalsize}
\emph{$^\dagger$Institute of Mathematics,
University of Warsaw, 02-097 Warsaw, Poland.} \\\vspace*{0.7em}
\emph{$^*$Department of Mathematical Sciences, Carnegie Mellon University, Pittsburgh, PA 15213, USA}
\end{normalsize}}

\begin{document}

\maketitle

\begin{abstract}
We sharpen the moment comparison inequalities with sharp constants for sums of random vectors uniform on Euclidean spheres, providing a deficit term (optimal in high dimensions). 
\end{abstract}

\bigskip

\begin{footnotesize}

\end{footnotesize}

\bigskip

\section{Introduction}

Sharp moment comparison inequalities (a.k.a. of the Khinchin-type, \cite{Kh}) have been extensively studied (see \cite{BC, BMNO, BN, CGT, CKT, Eitan, ENT-GM, ENT-Bpn, Haa, HNT, HT, Ko, KK, KLO, LO-best, LO, MRTT, NO, New, koles, Sz}), nonetheless the investigation of their stability presently appears to be in its nascent stages and has been focused so far only on the Rademacher sums (see \cite{BMNO, DDS, J}), as constituting, arguably, the most fundamental case. This note makes a first step towards widening the scope of this investigation and is devoted to sharpening sharp moment comparison inequalities for sums of random vectors uniform on Euclidean spheres, which provide a natural compelling generalisation of the Rademacher distribution to Euclidean space.

\subsection{New results}

Cutting to the chase, our main results read as follows. (We work in $\R^d$, equipped with the standard inner product $\scal{x}{y} = \sum_{j=1}^n x_jy_j$, $x, y \in \R^d$, and the endowed Euclidean norm $|x| = \sqrt{\scal{x}{x}}$.)

\begin{theorem}\label{thm:main}
Let $d \geq 2$ be a fixed dimension, let $\xi_1, \xi_2, \dots$ be independent random vectors uniform on the unit Euclidean sphere $S^{d-1}$ in $\R^d$ and let $Z$ be a Gaussian random vector in $\R^d$ with mean $0$ and covariance $\frac{1}{d}I_d$. Let $p \geq 2$. For every $n \geq 1$, every real scalars $a_1, \dots, a_n$ with $\sum_{j=1}^n a_j^2 = 1$, we have
\begin{equation}\label{eq:LpL2-deficit}
\E\left|\sum_{j=1}^n a_j\xi_j\right|^p \leq \E|Z|^p - c_{p,d}\sum_{j=1}^n a_j^4,
\end{equation}
where
\[ 
c_{p,d} = \frac{(p+d-2)(p+d-4)}{24d^2(d+2)}\cdot \begin{cases} 3p(p-2), & 2 \leq p \leq 4, \\ 1, & p > 4. \end{cases}
 \]
\end{theorem}

\emph{Remark.}
For a fixed $p$, as $d \to \infty$,  we have $c_{p,d} = \Theta_p(1/d)$. This is best possible, since the special case $n=1$ gives the bound
\[ 
c_{p,d} \leq \E|Z|^p - 1 = \frac{\Gamma(\frac{p+d}{2})}{(\frac{d}{2})^{p/2}\Gamma(\frac{d}{2})}-1 = O_p(1/d),
 \]
as can be checked for instance by a direct calculation using that $d\cdot |Z|^2$ follows the chi-squared distribution $\chi^2(d)$ with $d$-degrees of freedom which has density $\frac{1}{2^{d/2}\Gamma(d/2)}x^{d/2-1}e^{-x/2}$ on $(0, +\infty)$.

We emphasise that the Gaussian distribution is normalised so that $\E|Z|^2 = 1 = \E|\xi|^2$, that is there is equality when $p=2$, so necessarily $c_{2,d} = 0$. When $d=1$, the $\xi_j$ are Rademacher random variables, that is random signs uniform on $\{-1, 1\}$ and the $Z_j$ are standard Gaussian random variables $N(0,1)$, and when $d=2$, the $\xi_j$ are Steinhaus random variables (upon the usual identification $\R^2 \simeq \C$), both distributions playing a pivotal role in Banach space theory, see \cite{Ver, LT, Woj}.

For coefficient vectors $a = (a_1, \dots, a_n)$ of a \emph{fixed} length $n$, we offer the following stability result.

\begin{theorem}\label{thm:main-diag}
Let $d \geq 2$ be a fixed dimension, let $\xi_1, \xi_2, \dots$ be independent random vectors uniform on the unit Euclidean sphere $S^{d-1}$ in $\R^d$. Let $p \geq 2$. For every $n \geq 2$, every real scalars $a_1, \dots, a_n$ with $\sum_{j=1}^n a_j^2 = 1$, we have
\begin{equation}\label{eq:LpL2-deficit-diag}
\E\left|\sum_{j=1}^n a_j\xi_j\right|^p \leq \E\left|\sum_{j=1}^n \frac{1}{\sqrt{n}}\xi_j\right|^p - \tilde{c}_{p,d}\sum_{j=1}^n \left(\frac{1}{n}-a_j^2\right)^2,
\end{equation}
where
\[ 
\tilde{c}_{p,d} = \frac{p(p-2)}{4d}\begin{cases}
\frac{(p+d-2)(p+d-4)}{d(d+2)}, & 2 \leq p \leq 4, \\ 0.385, & p > 4. \end{cases}
 \]
\end{theorem}

%

\subsection{Previous work}

Inequalities \eqref{eq:LpL2-deficit} and \eqref{eq:LpL2-deficit-diag} have been previously proved without the deficit terms $O_{p,d}(\sum_{j=1}^n a_j^4)$ and $O_{p,d}(\sum_{j=1}^n (1/n-a_j^2)^2)$ respectively, by K\"onig and Kwapie\'n in \cite{KK}, and, independently, by Baernstein II and Culverhouse in \cite{BC}, who de facto established a more general convexity-type result which in particular asserts that the function
$
(x_1, \dots, x_n) \mapsto \E\left|\sqrt{x_j}\xi_j\right|^p
$
is Schur-concave on $\R_+^n$ when $p \geq 2$ (see \cite{BC} for background and details).
The inequality is often restated equivalently in a homogeneous form as the sharp $L_p-L_2$ moment comparison inequality,
\begin{equation}\label{eq:LpL2}
\E\left|\sum_{j=1}^n a_j\xi_j\right|^p \leq (\E|Z|^p)\left(\E\left|\sum_{j=1}^n a_j\xi_j\right|^2\right)^{p/2}, \qquad p \geq 2,
 \end{equation}
for all $a_1, \dots, a_n \in \R$. The multiplicative constant $\E|Z|^p$ is sharp, as follows from the case $a_1 = \dots = a_n = \frac{1}{\sqrt{n}}$, $n \to \infty$ in view of the central limit theorem. 

As hinted earlier, there are few stability results, and only for Rademacher sums, that is in the case $d=1$, when the $\xi_j$ are i.i.d. uniform random signs. In this classical setting, inequalities \eqref{eq:LpL2-deficit} and \eqref{eq:LpL2-deficit-diag} are known to hold for all $p \geq 3$, as has been recently established by Jakimiuk in \cite{J}, albeit with a worse dependence on $p$ of the constant $c_{p,1}$ for large $p$; see also Corollary 1 in \cite{BMNO}, where this is derived as a by product of the sharp $L_p-L_4$ inequality for $p \geq 4$.  Prior to Jakimiuk's work, there had been one more stability result, viz. De, Diakonikolas and Servedio in \cite{DDS} found a deficit term in the celebrated Szarek's $L_1-L_2$ inequality from \cite{Sz} (see also \cite{ENT-res} for a different approach and explicit constants). This was paralleled in \cite{CNT, GTW, M3T} in the geometric context of stability results for maximal volume sections of  $\ell_p$-balls, polydisc and simplex, respectively (which themselves can be viewed as the moment comparison inequalities but in ``$L_{-1}$''), and has found interesting applications, see \cite{ENT-res, MR}. In a different spirit, ``distributional stability'' has been recently investigated in \cite{ENT-stab}.

We recommend, e.g. \cite{BMNO, J, NO} for further references on the pursuit of the sharp constants in the classical Khinchin inequality for random signs, as well as \cite{HT} for an account on what is known for other distributions, and \cite{CST} specifically for spherically symmetric random vectors.

It might be also worth making a point here that comparing our Theorem \ref{thm:main} valid for all $p \geq 2$ to Jakimiuk's result \cite{J} valid for $p \geq 3$ strongly attests to the paradigm that methods effective for Rademacher variables (the one-dimensional case $d=1$) often extend to a broader range of parameters in higher dimensions $d>1$. This perhaps goes back at least to Ullrich's works \cite{Ull1, Ull2} where he showed the equivalence between $L_0$ and $L_2$ norms  for Steinhaus random variables (case $d=2$), and emphasised that the analogous equivalence does not hold for Rademacher random variables (case $d=1$); another examples may include Hanner's inequality for many functions \cite{JT, Sch}, or optimal tail bounds \cite{BGH, CLT, NT, P1}.

\section{Proofs}

As in the statement of Theorems \ref{thm:main} and \ref{thm:main-diag}, throughout the rest of this paper, $\xi, \xi_1, \xi_2, \dots$ are independent identically distributed (i.i.d.) random vectors uniform on the unit sphere $S^{d-1} = \{x \in \R^d, \ |x| = 1\}$ in $\R^d$, and $Z, Z_1, Z_2, \dots$ are independent of them i.i.d. Gaussian random vectors in $\R^d$ with mean $0$ and covariance $\frac{1}{d}I_d$.

First, we focus solely on Theorem \ref{thm:main}. Second, having established and building on the important ideas and auxiliary lemmas, we shall prove Theorem \ref{thm:main-diag}.

\subsection{Overview}

At a high level, the main idea to tackle Theorem \ref{thm:main} is reminiscent of Lindeberg's swapping argument from his work \cite{Lin} on the central limit theorem, which has been widely used in a variety of contexts (see e.g. \cite{PR} for historical accounts), in particular for moment comparison inequalities (see, e.g. \cite{BN, CET, ENT-Bpn, HNT}). This has also been Jakimiuk's approach in \cite{J}.

Specifically, for $p > 0$, we define the deficit,
\begin{equation}\label{eq:def-deficit}
D_p(a, v) = \E|aZ + v|^p - \E|a\xi + v|^p, \qquad a \in \R, \ v \in \R^d.
\end{equation}
Suppose that it is nonnegative, $D_p(a,v) \geq 0$, for all $p \geq 2$, $a \in \R$ and $v \in \R^d$. Then, the proof of \eqref{eq:LpL2} goes by repeatedly swapping each $\xi_j$ with $Z_j$ (relying on the independence of the summands which allows in turn to condition on all but one summand that is being swapped),
\begin{align*}
\E\left|a_1\xi_1 + a_2\xi_2 + \dots + a_n\xi_n\right|^p &\leq \E\left|a_1Z_1 + a_2\xi_2 + \dots + a_n\xi_n\right|^p \leq \dots \\
&\leq \E\left|a_1Z_1 + a_2Z_2 + \dots + a_nZ_n\right|^p \\
&= \E|(a_1^2+\dots + a_n^2)^{1/2}Z|^p.
 \end{align*}
Now, to make some savings and obtain a deficit term in this bound, compellingly, we would like to sharpen the bound $D_p(a,v) \geq 0$.

This task brings us asking: why is $D_p(a,v)$ nonnegative (when $p \geq 2$)? Fix $a \in \R, v \in \R^d$ and consider the function
\[ 
h_{a, v}(t) = \E|v + at^{1/2}\xi|^p, \qquad t  > 0.
 \]
The heart of the matter in both \cite{BC} and \cite{KK} is the convexity of $h_{a,v}$ on $(0, +\infty)$, for then decomposing the distribution of $Z$ as $|Z|\xi$ (the magnitude and independent uniform direction), Jensen's inequality yields
\begin{align*}
\E|v + aZ|^p = \E|v + a\sqrt{|Z|^2}\xi|^p &=  \E_{|Z|} h_{a, v}(|Z|^2) \\
&\geq h_{a,v}(\E |Z|^2) = h_{a,v}(1) = \E|v + a\xi|^p,
 \end{align*}
which is $D_{p}(a,v) \geq 0$. 

As a side note, this argument is robust enough to treat arbitrary rotationally invariant random vectors whose magnitudes are comparable in the stochastic convex ordering, as done in \cite{KK}, as well as other functionals than just the moments, as in \cite{BC}.

\subsection{Main lemmas}

Our argument therefore begins with a derivation of an exact expression for the second derivative of functions $h_{a,v}$, amenable to quantitative improvements on their convexity. For greater transparency of the ensuing calculations, we follow \cite{BC} and treat arbitrary (smooth) functionals.

\begin{lemma}\label{lm:2nd-derivative}
Let $\Psi$ be a smooth function on $\R^d$. For a fixed vector $v \in \R^d$, define
\[ 
f(t) = \int_{S^{d-1}} \Psi\big(v + t^{1/2}x\big) \dd x, \qquad t > 0.
 \]
Then
\begin{align}
\label{eq:f'}
f'(t) &= \frac{1}{2}\int_{B_2^d}\Delta\Psi\big(v + t^{1/2}x\big) \dd x, \\\label{eq:f''}
f''(t) &= \frac{1}{4}\int_0^1r^{d+1}\left(\int_{B_2^d}\Delta\Delta\Psi\big(v + t^{1/2}rx\big) \dd x\right)\dd r.
\end{align}
\end{lemma}
\begin{proof}
The approach we use was indicated in Remark 15 in \cite{CST} and is in the spirit of Lemma 4 from \cite{BC}. Plainly,
\[ 
f'(t) = \frac12t^{-1/2}\int_{S^{d-1}}\scal{\nabla \Psi\big(v + t^{1/2}x\big)}{x} \dd x.
 \]
Since $x$ is the outer-normal unit vector, the divergence theorem gives
\[ 
f'(t) = \frac12t^{-1/2}\int_{B_2^{d}} \text{div}_x\Big(\nabla \Psi\big(v + t^{1/2}x\big)\Big) \dd x =  \frac12\int_{B_2^{d}} \Delta\Psi\big(v + t^{1/2}x\big) \dd x.
 \]
To find the next derivative, note that using polar coordinates,
\[ 
f'(t) = \frac12\int_0^1r^{d-1}\left(\int_{S^{d-1}}  \Delta\Psi\big(v + t^{1/2}rx\big) \dd x\right)\dd r
 \]
and the point is that the integral over the sphere, as the function in $t$ is of the same form as $f$, with the variable rescaled  by $r^2$. Therefore, applying the previous calculation, we get
\[ 
f''(t) = \frac12\int_0^1 r^{d-1}\left(\frac12\int_{B_2^{d}}  \Delta\Delta\Psi\big(v + t^{1/2}rx\big)\cdot r^2 \dd x\right)\dd r.\qedhere
 \]
\end{proof}

The workhorse of our proof will be the following quantitative bound on the deficit introduced in \eqref{eq:def-deficit}.

\begin{lemma}\label{lm:Dp}
Let $p \geq 2$. For $a \in \R, v \in \R^d$, $d \geq 2$, we have
\[ 
D_p(a, v) \geq \kappa_{p,d}a^4\cdot\begin{cases}
(|v|^2+2a^2)^{\frac{p-4}{2}}, & 2 \leq p \leq 4, \\
|v|^{p-4}, & p > 4,
\end{cases}
 \]
where 
\[
\kappa_{p,d} = \frac{p(p-2)(p+d-2)(p+d-4)}{4d^2(d+2)}.
\]
\end{lemma}
\begin{proof}
If $a = 0$, there is equality. Otherwise, by homogeneity, we can assume without loss of generality that $a=1$. Fix a vector $v$ in $\R^d$ and consider the function
\[ 
g_v(t) = \E|v + \sqrt{t}\xi|^p, \qquad t > 0.
 \]
By the rotational invariance of the Gaussian distribution, $Z$ has the same distribution as $|Z|\xi$. Thus,
\[ 
D_p(1, v) = \E\Big[g_v(|Z|^2) - g_v(1)\Big].
 \]
Using Taylor's expansion with Lagrange's remainder, for $t > 0$, there is $\eta_t$ between $1$ and $t$ such that
\[ 
g_v(t) - g_v(1) = (t-1)g_v'(1) + \frac{1}{2}(t-1)^2g_v''(\eta_t).
 \]
As a result (recall $\E|Z|^2 = 1$),
\begin{equation}\label{eq:Dp}
    D_p(1, v) = \frac{1}{2}\E\Big[(|Z|^2-1)^2g_v''(\eta_{|Z|^2})\Big].
\end{equation}
An application of \eqref{eq:f''} of Lemma \ref{lm:2nd-derivative} with $\Psi(x) = |x|^p$ yields for every $\eta > 0$,
\[ 
g_v''(\eta) = \frac{p(p-2)(p+d-2)(p+d-4)}{4|S^{d-1}|}\int_0^1 r^{d+1}\left(\int_{B_2^d}|v + xr\sqrt{\eta}|^{p-4}\dd x\right)\dd r
 \]
(sweeping under the rug inessential regularity issues caused by the singularity at the origin, overcome e.g. by taking the smooth approximations $\Psi_\delta(x) = (|x|^2+\delta)^{p/2}$, $\delta \downarrow 0$, see \cite{BC}, eq. (1.5) for details). Writing the integral over $B_2^d$ in polar coordinates leads to
\begin{equation}\label{eq:g''}
g_v''(\eta) = \beta_{p,d}\int_0^1\int_0^1 \E_\xi |v + r_1r_2\sqrt{\eta}\xi|^{p-4} \dd r_1^d\dd r_2^{d+2}, \qquad \eta > 0,
\end{equation}
with 
\begin{equation}\label{eq:beta-def}
\beta_{p,d} = \frac{p(p-2)(p+d-2)(p+d-4)}{4d(d+2)}. 
\end{equation}
Note that for convenience, we have renormalised the double integral $\int_0^1\int_0^1$ so that it is against the probability measure on $[0,1]^2$ with density $d(d+2)r_1^{d-1}r_2^{d+1}$. 

Our argument lower bounding $g''$ now differs depending on whether $p > 4$.

\emph{Case 1: $2 \leq p \leq 4$.} Let $q = \frac{p-4}{2} \leq 0$. Here, we lean on the convexity of the function $x \mapsto x^q$ on $(0,+\infty)$. Rewriting \eqref{eq:g''} and using Jensen's inequality,
\begin{align}\notag
g_v''(\eta) &= \beta_{p,d}\int_0^1\int_0^1\E_\xi \big(|v|^2 + 2r_1r_2\sqrt{\eta}\scal{v}{\xi} + r_1^2r_2^2\eta\big)^q \dd r_1^d\dd r_2^{d+2} \\\label{eq:g''-lower-bd}
&\geq \beta_{p,d}\left(|v|^2 + \frac{d}{d+4}\eta\right)^q.
\end{align}
Plugging this back into \eqref{eq:Dp} and using a crude bound $\eta_{|Z|^2} \leq \max\{1,|Z|^2\}$, we arrive at
\[ 
D_p(1, v) \geq \frac{1}{2}\beta_{p,d}\E\left[(|Z|^2-1)^2\left(|v|^2 + \frac{d}{d+4}\max\{1,|Z|^2\}\right)^q\right].
 \]
To finish off with a clean bound, we use Jensen's inequality yet again with respect to the probability measure $\frac{(|Z|^2-1)^2}{2/d}\dd \mathbb{P}$ and obtain
\[ 
D_p(1, v) \geq  \frac{1}{d}\beta_{p,d}\left(|v|^2 + \frac{d}{d+4}\E\left[\max\{1,|Z|^2\}\frac{(|Z|^2-1)^2}{2/d}\right]\right)^q.
 \]
Finally, crudely $\max\{1,|Z|^2\} \leq 1 + |Z|^2$, thus
\[ 
\E\left[\max\{1,|Z|^2\}\frac{(|Z|^2-1)^2}{2/d}\right] \leq \frac{d}{2}\left(\E(|Z|^2-1)^2+\E\Big[|Z|^2(|Z|^2-1)^2\Big]\right) = 1 + \frac{d+4}{d},
 \]
where the last two expectations are calculated directly using that $d\cdot |Z|^2$ follows the $\chi^2(d)$ distribution. Consequently,
\[ 
D_p(1, v) \geq \frac{1}{d}\beta_{p,d} \left(|v|^2 + \frac{d}{d+4}+1\right)^q \geq \frac{1}{d}\beta_{p,d} \left(|v|^2 +2\right)^q.
 \]

\emph{Case 2: $p > 4$.} We simply use monotonicity asserted by the following immediate consequence of \eqref{eq:f'}.

\textbf{Claim.}
Let $q > 0$, $v \in \R^d$, $d \geq 2$. Then
\[
t \mapsto \E|v+\sqrt{t}\xi|^q \qquad \text{increases on $(0,+\infty)$.}
\]
In particular,
\begin{equation}\label{eq:monot-lower-bd}
\E|v+\sqrt{t}\xi|^q \geq |v|^q,
\end{equation}
and, thanks to rotational invariance,
\begin{equation}\label{eq:monot}
(v,t) \mapsto \E|v+\sqrt{t}\xi|^q = \E\big||v|\xi' + \sqrt{t}\xi\big|^q \quad  \text{increases both in $t$ and $|v|$.}
\end{equation}
Indeed, by \eqref{eq:f'},
\[ 
\frac{\dd}{\dd t}\E|v+\sqrt{t}\xi|^q = \frac{q(q+d-2)}{2|S^{d-1}|}\int_{B_2^d} |v+\sqrt{t}x|^{q-2} \dd x \geq 0. \qedhere
 \]
Therefore, using \eqref{eq:g''},
\[ 
g_v''(\eta) \geq \beta_{p,d}|v|^{p-4}
\]
(deterministically, for every $\eta > 0$). Plugging this pointwise bound into \eqref{eq:Dp} gives
\[ 
D_p(1, v) \geq \frac{1}{2}\E\Big[(|Z|^2-1)^2\Big]\beta_{p,d}|v|^{p-4}.
\]
Since
$
\E\Big[(|Z|^2-1)^2\Big] = \frac{2}{d},
$
we obtain $D_p(1,v) \geq \frac{1}{d}\beta_{p,d}|v|^{p-4}$, that is the lemma holds with
$
\kappa_{p,d} = \frac{1}{d}\beta_{p,d},
$
as desired.
\end{proof}

\subsection{An auxiliary lemma}

In order to handle the averages of the terms $\E|v|^{p-4}$ coming from the bound on $D_p$ in the case $p > 4$, with $v$ being sums of uniforms of spheres, we shall need some sort of concentration. For simplicity, we choose to use Khinchin-type inequalities (which in fact yields explicit and decent values of the constants involved). 

\begin{lemma}\label{lm:Khin}
Let $d \geq 2$. There is a universal constant $c_{\mathrm{Kh}} > 0$ such that for all real numbers $a_1, a_2, \dots$ and $q > 0$, we have
\[ 
\E\left|\sum_{j=1}^n a_j\xi_j\right|^q \geq c_{\mathrm{Kh}}\left(\sum_{j=1}^n a_j^2\right)^{q/2}.
 \]
One can take $c_{\mathrm{Kh}} = 0.77$. In particular, the same inequality holds if any of the variables $\xi_j$ is replaced by $Z_j$.
\end{lemma}
\begin{proof}
When $q \geq 2$, the inequality plainly holds with constant $1$ (by Jensen's inequality). When $0 < q < 2$, let $c_{d,q}$ be the best constant such that the Khinchin-type inequality
\[ 
\E\left|\sum_{j=1}^n a_j\xi_j\right|^q \geq c_{d,q}\left(\sum_{j=1}^n a_j^2\right)^{q/2}
 \]
holds for all $n \geq 1$ and all scalars $a_j$. It is the main result of \cite{CGT, Ko, KK} that
\[ 
c_{d, q} = \min\{ 2^{-q/2}\E|\xi_1+\xi_2|^q, \E|Z|^q \}.
 \]
(see also \cite{CST}), and it is known that when $d \geq 3$, this minimum is attained at the second term, $\E|Z|^q$, which we now lower bound. We have,
\[
\E|Z|^q = \frac{\Gamma(\frac{q+d}{2})}{\left(\frac{d}{2}\right)^{q/2}\Gamma(\frac{d}{2})}.
\]
By the log-convexity of the Gamma function, for $x > 0$ and $0 < s < 1$,
\begin{equation}\label{eq:Wen}
\begin{split}
\Gamma(x+1) &= \Gamma\big(s(x+s) + (1-s)(x+s+1)\big) \\
&\leq \Gamma(x+s)^s\Gamma(x+s+1)^{1-s} 
= (x+s)^{1-s}\Gamma(x+s)
\end{split}
 \end{equation}
which is Wendel's inequality, \cite{Wen}, resulting in
\[ 
\frac{\Gamma(x+s)}{x^s\Gamma(x)} \geq \left(\frac{x}{x+s}\right)^{1-s}.
 \]
Applied to $x = \frac{d}{2} \geq 1$, $s = \frac{q}{2}$, we obtain
\[ 
\E|Z|^q \geq \frac{1}{(1+s)^{1-s}} \geq e^{-s(1-s)} \geq e^{-1/4} > 0.778.
 \]
It remains to lower bound the first term  $2^{-q/2}\E|\xi_1+\xi_2|^q$ when $d = 2$. We have,
\[ 
2^{-q/2}\E|\xi_1+\xi_2|^q = 2^{q/2}\frac{\Gamma(\frac{q}{2}+\frac{1}{2})}{\sqrt{\pi}\Gamma(\frac{q}{2}+1)}.
 \]
Again, by virtue of \eqref{eq:Wen}, applied this time with $x = \frac{q}{2}$, $s = \frac12$, we get
\[ 
2^{-q/2}\E|\xi_1+\xi_2|^q \geq \frac{2^{\frac{q+1}{2}}}{\sqrt{\pi(q+1)}}.
 \]
The right hand side is minimised at $q = \frac{1}{\log 2} - 1$ attaining value $0.774..$, which finishes the proof.
\end{proof}

\subsection{Proof of Theorem \ref{thm:main}}
We shall follow the traditional notation $\|a\|_p = (\sum |a_j|^p)^{1/p}$, $\|a\|_\infty = \max_j |a_j|$ for the $\ell_p$ norms of a vector $a = (a_1, \dots, a_n)$ in $\R^n$.

The proof uses two different arguments: when $\|a\|_\infty$ is bounded away from $1$, we shall (iteratively) use the pointwise bounds from Lemma \ref{lm:Dp} resulting with the deficit of the order $\|a\|_4$, whereas in the oppose case we easily get a constant deficit (i.e. independent of $a$), leveraging the Schur concavity of the moment functional. With hindsight, we choose the following cut-off for the $\ell_\infty$ norm,
\[ 
m_p = \begin{cases}
1, & 2 \leq p \leq 4,\\
\sqrt{1 - 2^{-\frac{1}{p-4}}}, & p > 4.
\end{cases}
 \]

\emph{Case 1: $\|a\|_\infty \leq m_p$.}
(Clarification: when $2 \leq p \leq 4$ this case is exhaustive, since $m_p = 1$.) We use the classical Lindeberg's swapping argument. To this end, we define
\begin{align*}
S_0 &= \sum_{j=1}^n a_j\xi_j, \\
S_k &= \sum_{j=1}^k a_jZ_j + \sum_{j=k+1}^n a_j\xi_j, \qquad k = 1, \dots, n
\end{align*}
and break the deficit up with a telescoping sum,
\[ 
\E|Z|^p - \E|S_0|^p = \sum_{k=1}^n \big(\E|S_{k}|^p - \E|S_{k-1}|^p\big).
 \]
Note that the sums $S_{k}$ and $S_{k-1}$ only differ by the $k$-th term which is $a_{k}Z_k$ and $a_{k}\xi_k$, respectively. Letting 
\[ 
v_k = \sum_{j=1}^{k-1} a_jZ_j + \sum_{j=k+1}^n a_j\xi_j
 \]
and using the notation from \eqref{eq:def-deficit}, we have
\[ 
\E|S_{k+1}|^p - \E|S_k|^p = \E D_p(a_k,v_k).
 \]
By Lemma \ref{lm:Dp}, 
\[ 
\E D_p(a_k,v_k) \geq \kappa_{p,d}a_k^4\begin{cases}
\E(|v_k|^2+2a_k^2)^{\frac{p-4}{2}}, & 2 \leq p \leq 4, \\
\E|v_k|^{p-4}, & p > 4.
\end{cases}
 \]
Observe that $\E|v_k|^2 = \sum_{j \neq k} a_j^2 = 1-a_k^2$. 

When $2 \leq p \leq 4$, Jensen's inequality yields
\[ 
\E D_p(a_k,v_k) \geq \kappa_{p,d}a_k^4(1 - a_k^2 + 2a_k^2)^{\frac{p-4}{2}} = \kappa_{p,d}a_k^4(1+a_k^2)^{\frac{p-4}{2}} \geq \frac12\kappa_{p,d}a_k^4.
 \]
Summing these bounds over $1 \leq k \leq n$ gives the result and finishes the proof.

When $p > 4$, Lemma \ref{lm:Khin} and a crude bound $1 - a_k^2 \geq 1 - \|a\|_\infty^2 \geq 1 - m_p^2 = 2^{-\frac{1}{p-4}}$ yield
\[ 
\E D_p(a_k,v_k) \geq c_{\mathrm{Kh}}\kappa_{p,d}a_k^4(1-a_k^2)^{\frac{p-4}{2}} \geq 2^{-1/2}c_{\mathrm{Kh}}\kappa_{p,d}a_k^4 > \frac{1}{2}\kappa_{p,d}a_k^4.
\]
Summing these bounds over $1 \leq k \leq n$ gives the result. 

\emph{Case 2: $\|a\|_\infty > m_p$.} 
Clarification: when $2 \leq p \leq 4$, $m_p = 1$, this case is empty, and the proof has already been completed, so we implicitly assume that $p > 4$. Here we simply use the Schur concavity of
\[ 
\R_+^n \ni x \mapsto \E\left|\sum_{j=1}^n \sqrt{x_j}\xi_j\right|^p
 \]
known from \cite{BC} to hold in every dimension $d \geq 2$ as long as $p \geq 2$. Say $a_1 = \|a\|_\infty$. Then the vector $(a_1^2, \dots, a_n^2)$  majorises the vector $(m_p^2, \frac{1-m_p^2}{n-1}, \dots,  \frac{1-m_p^2}{n-1})$, provided that $m_p^2 \geq  \frac{1-m_p^2}{n-1}$, equivalently $nm_p^2 \geq 1$, and we obtain
\[ 
\E\left|\sum_{j=1}^n a_j\xi_j\right|^p \leq \E\left|m_p\xi_1  +\sqrt{\frac{1-m_p^2}{n-1}}(\xi_2 + \dots + \xi_{n-1})\right|^p.
 \]
We can apply Case 1 to the right hand side, which results in
\[ 
\E\left|\sum_{j=1}^n a_j\xi_j\right|^p \leq \E|Z|^p - \frac{1}{2}\kappa_{p,d} \left( m_p^4 + \frac{(1-m_p^2)^2}{n-1}\right) \leq \E|Z|^p - \frac{1}{2}\kappa_{p,d}m_p^4
 \]
Note that $\|a\|_4 \leq \|a\|_2 = 1$. For cosmetics, say $m_p^4 = (1-2^{-1/(p-4)})^2 > \frac{1}{3p(p-2)}$ which gives the constant $c_{p,d}$ from the statement of the theorem. Finally, if $nm_p^2 < 1$, we take an integer $N \geq 2$ such that $N \geq \frac{1}{m_p^2}$ and observe that $(a_1^2, \dots, a_n^2, \underbrace{0, \dots, 0}_{N-n})$  majorises the vector $(m_p^2, \frac{1-m_p^2}{N-1}, \dots,  \frac{1-m_p^2}{N-1})$, so that we can repeat the last part of the argument verbatim to finish the proof. \hfill\qed

\subsection{Proof of Theorem \ref{thm:main-diag}}

Exactly as for Theorem \ref{thm:main}, the argument here will be driven by tracking the deficit along \emph{local} changes to the coefficient vector $a = (a_1, \dots, a_n)$ performed to make it progressively closer to the extremising diagonal one $(\frac{1}{\sqrt{n}}, \dots, \frac{1}{\sqrt{n}})$. The next lemma will facilitate that. We denote the deficit term from Theorem~\ref{thm:main-diag} by $\delta(a)$,
\[ 
\delta(a) = \sum_{j=1}^n \left(\frac{1}{n}-a_j^2\right)^2
 \]
and note that with the $\ell_2$ normalisation $\sum_{j=1}^n a_j^2 = 1$,
\begin{equation}\label{eq:delta-diag}
\delta(a) = \sum_{j=1}^n a_j^4 - \frac{1}{n} =  \sum_{j=1}^n a_j^4 - \sum_{j=1}^n\frac{1}{\sqrt{n}^4},
 \end{equation}
that is the deficit term can be equivalently derived by comparing the changes of the $\ell_4$ norms of the coefficient vectors.

\begin{lemma}\label{lm:diag-local}
Let $p \geq 2$. Suppose that  $a_1 \geq b_1 \geq b_2 \geq a_2 > 0$ with $a_1^2 + a_2^2 = b_1^2 + b_2^2 = \sigma^2$ for some $\sigma > 0$. For every vector $v$ in $\R^d$, $d \geq 2$, we have
\begin{align}\notag
&\E\left|b_1\xi_1 + b_2\xi_2 + v\right|^p - \E\left|a_1\xi_1 + a_2\xi_2 + v\right|^p \\\label{eq:diag-local}
&\qquad\qquad\geq \frac{1}{d}(a_1^4+a_2^4 - b_1^4-b_2^4)\cdot \begin{cases}\beta_{p,d}(|v|^2+\sigma^2)^{\frac{p-4}{2}}, & 2 \leq p \leq 4, \\ \frac{p(p-2)}{8}\E|v+\sigma\xi|^{p-4}, & p > 4, \end{cases}
\end{align}
with $\beta_{p,d}$ defined in \eqref{eq:beta-def}.
\end{lemma}
\begin{proof}
Recall that Lemma \ref{lm:2nd-derivative} provides us with good expressions for the derivatives of the function
\[ 
g_v(t) = \E|v + \sqrt{t}\xi|^p, \qquad t \geq 0
 \]
and this very function emerges naturally: thanks to independence and rotational invariance of the ensuing random vectors,
\begin{align*}
\E|a_1\xi_1 + a_2\xi_2 + v|^p = \E\big||a_1\xi_1 + a_2\xi_2|\xi + v\big|^p &= \E g_v\big(|a_1\xi_1+a_2\xi_2|^2\big) \\
&= \E g_v\big(\sigma^2 + 2a_1a_2\theta\big),
\end{align*}
where the last expectation is with respect to a random variable $\theta$  which has the same distribution as a one-dimensional marginal of $\xi$, say $\theta = \scal{\xi}{e_1}$. With $\sigma > 0$ fixed, we set
\[ 
h(u) = \E \big[g_v(\sigma^2 + 2u\theta)\big], \qquad u > 0.
 \]
Then the deficit of interest becomes
\begin{align*}
\Delta = \E\left|b_1\xi_1 + b_2\xi_2 + v\right|^p - \E\left|a_1\xi_1 + a_2\xi_2 + v\right|^p &= h(b_1b_2) - h(a_1a_2) \\
&= \int_{a_1a_2}^{b_1b_2} h'(u) \dd u.
\end{align*}
Plainly,
\[ 
h'(u) = 2\E\big[g_v'(\sigma^2 + 2u\theta)\theta\big].
 \]
To manoeuvrer this into a more convenient expression, we shall use an integration by parts formula for $\theta$.

\textbf{Claim.} Let $d \geq 2$ and $\xi, \tilde{\xi}$ be uniform on $S^{d-1}$, $S^{d+1}$ respectively. For random variables $\theta = \scal{\xi}{e_1}$, $\tilde \theta = \scal{\tilde \xi}{e_1}$ and a function $f$ differentiable on $(-1,1)$ such that $f(x)(1-x^2)^{\frac{d-1}{2}} \to 0$ as $x \to \pm 1$, we have
\begin{equation}\label{eq:theta-by-parts}
\E\big[f(\theta)\theta\big] = \frac{1}{d}\E\big[f'(\tilde \theta)\big].
\end{equation}
\begin{proof}
One can check that $\theta$ has density $\frac{1}{A_d}(1-x^2)^{\frac{d-3}{2}}\1_{(-1,1)}(x)$ with the normalising constant $A_d = \sqrt{\pi}\frac{\Gamma(\frac{d-1}{2})}{\Gamma(\frac{d}{2})}$, $d \geq 2$. Integration by parts yields
\begin{align*}
\E\big[f(\theta)\theta\big] &= \frac{1}{A_d}\int_{-1}^1 f(x)x(1-x^2)^{\frac{d-3}{2}} \dd x = \frac{1}{A_d}\int_{-1}^1 f(x)\Big(-\frac{1}{d-1}(1-x^2)^{\frac{d-1}{2}}\Big)' \dd x \\
&= -\frac{1}{(d-1)A_d}f(x)(1-x^2)^{\frac{d-1}{2}}\Bigg|_{-1}^1 + \frac{1}{(d-1)A_d}\int_{-1}^1 f'(x)(1-x^2)^{\frac{d-1}{2}} \dd x \\
&= \frac{A_{d+2}}{(d-1)A_d}\E\big[f'(\tilde \theta)\big]
\end{align*}
and $\frac{A_{d+2}}{(d-1)A_d} = \frac{1}{d}$, as claimed. \qedhere
\end{proof}

Applying the claim to $f(x) = g_v'(\sigma^2 + 2ux)$ gives
\[ 
h'(u) = \frac{2}{d}\E\Big[g_v''(\sigma^2 + 2u\tilde \theta)\cdot (2u)\Big]
 \]
where the expectation is over the distribution of $\tilde \theta$ from the statement of the claim. Moreover, evoking \eqref{eq:g''},
\[ 
g_v''(t) = \beta_{p,d}\E |v+R_1R_2\sqrt{t}\xi|^{p-4},
 \]
where, to compactify the notation, we let $(R_1,R_2)$ be random variables with joint density $\dd r_1^d\dd r_2^{d+2}$ on $(0,1)^2$.

Putting these together,
\[ 
h'(u) = \tilde\beta_{p,d}\E\big[|v+R_1R_2(\sigma^2 + 2u\tilde \theta)^{1/2}\xi|^{p-4}\big]\cdot (2u)
 \]
with 
\[
\tilde \beta_{p,d} = \frac{2}{d}\beta_{p,d}
\]
and the expectation taken over the product distribution of $R_1, R_2, \tilde \theta, \xi$.
Thus, we arrive at
\begin{align*}
\Delta = \tilde \beta_{p,d}\int_{a_1a_2}^{b_1b_2} \E\big[|v+R_1R_2(\sigma^2 + 2u\tilde \theta)^{1/2}\xi|^{p-4}\big]\cdot (2u)\dd u.
\end{align*}
As before, we now break the analysis into two cases depending on whether $p > 4$.

\emph{Case 1: $2 \leq p \leq 4$.} As in Lemma \ref{lm:Dp}, letting $q = \frac{p-4}{2}$ and using the point-wise bound \eqref{eq:g''-lower-bd}, we get
\[ 
\Delta \geq \tilde\beta_{p,d}\int_{a_1a_2}^{b_1b_2} (2u)\cdot \E\left[\left(|v|^2 + \frac{d}{d+4}(\sigma^2 + 2u\tilde\theta)\right)^q\right] \dd u.
 \]
Since $\E\tilde\theta = 0$, Jensen's inequality and a further simple cosmetic bound $\frac{d}{d+4}<1$ allow to lower-bound the expectation by $(|v|^2+\sigma^2)^q$ which results in
\[ 
\Delta \geq  \tilde\beta_{p,d}(|v|^2+\sigma^2)^q(b_1^2b_2^2-a_1^2a_2^2).
 \]
Finally, as a result of the constraint $a_1^2+a_2^2 = b_1^2+b_2^2$,
\[ 
0 = (b_1^2+b_2^2)^2 - (a_1^2+a_2^2)^2 = b_1^4+b_2^4 - (a_1^4+a_2^4) + 2(b_1^2b_2^2-a_1^2a_2^2),
 \]
so
\begin{equation}\label{eq:ab-prod-to-sum}
b_1^2b_2^2-a_1^2a_2^2 = \frac12(a_1^4+a_2^4-b_1^4-b_2^4)
\end{equation}
and we are done with the proof in this case with constant $\frac12 \tilde\beta_{p,d}$, as desired.

\emph{Case 2: $p > 4$.} We can crudely bound the expectation by restricting it to the positive values of $\tilde\theta$,
\begin{align*}
\Delta \geq \tilde \beta_{p,d}\int_{a_1a_2}^{b_1b_2} \E\big[|v+R_1R_2(\sigma^2 + 2u\tilde \theta)^{1/2}\xi|^{p-4}\1_{\{\tilde\theta > 0\}}\big]\cdot (2u)\dd u
\end{align*}
Thanks to \eqref{eq:monot}, used conditioning on the positive values of $R_1, R_2, \tilde\theta$,
\begin{align*}
\Delta &\geq  \tilde \beta_{p,d}\int_{a_1a_2}^{b_1b_2} \E\big[|v+R_1R_2\sigma\xi|^{p-4}\big]\cdot \p{\tilde\theta > 0}\cdot  (2u)\dd u \\
&= \frac{1}{2}\tilde \beta_{p,d}\E\big[|v+R_1R_2\sigma\xi|^{p-4}\big](b_1^2b_2^2-a_1^2a_2^2).
\end{align*}
Since $R_1R_2 \leq 1$ a.s., using \eqref{eq:monot} once more, we get a bound
\[ 
\E\big[|v+R_1R_2\sigma\xi|^{p-4}\big]  \geq \E\big[|R_1R_2v+R_1R_2\sigma\xi|^{p-4}\big] = \E(R_1R_2)^{p-4}\E\big[|v+\sigma\xi|^{p-4}\big].
 \]
Consequently, employing \eqref{eq:ab-prod-to-sum},
\[
\Delta \geq \frac{1}{4}\tilde\beta_{p,d}\E(R_1R_2)^{p-4}\E\big[|v+\sigma\xi|^{p-4}\big](a_1^4+a_2^4-b_1^4-b_2^4).
\]
By a direct calculation, $\E(R_1R_2)^{p-4} = \frac{d(d+2)}{(p-4+d)(p-2+d)}$, so after tracing the constants and simplifying, we have
\[ 
\frac{1}{4}\tilde\beta_{p,d}\E(R_1R_2)^{p-4} = \frac{p(p-2)}{8d}.\qedhere
 \]
\end{proof}

\begin{proof}[Proof of Theorem \ref{thm:main-diag}]
When $n=2$, we apply Lemma \ref{lm:diag-local} with $v=0$ directly to the coefficient vectors $b = (\frac{1}{\sqrt2}, \frac{1}{\sqrt2})$ and $a = (a_1, a_2)$ for which $\sigma = 1$. We get
\[ 
\E\left|\frac{1}{\sqrt2}\xi_1+\frac{1}{\sqrt2}\xi_2\right|^p - \E|a_1\xi_1 + a_2\xi_2|^p \geq \tilde c_{p,d}\delta(a),
 \]
as desired in view of \eqref{eq:delta-diag}, where
\[ 
\tilde c_{p,d} = \frac{1}{d}\cdot \begin{cases}\beta_{p,d}, & 2 \leq p \leq 4, \\ \frac{p(p-2)}{8}, & p > 4. \end{cases}
 \]
Now suppose that $n \geq 3$, $a_1 \geq \dots \geq a_n > 0$ and that $a \neq (\frac{1}{\sqrt{n}}, \dots, \frac{1}{\sqrt{n}})$ (in particular, $a_n < \frac{1}{\sqrt{n}}$). We follow a strategy from \cite{J}, namely we repetitively perform the following local operation on the coefficient vector $a$ until it becomes the diagonal vector $(\frac{1}{\sqrt{n}}, \dots, \frac{1}{\sqrt{n}})$: we take its largest coefficient $a_1$, the smallest one $a_n$, replace them with $a_1' = \sqrt{a_1^2+a_n^2 - \frac{1}{n}}$ and $a_n' = \frac{1}{\sqrt{n}}$, and finally rearrange the coefficients of $(a_1', a_2, \dots, a_{n-1}, a_n')$ to be nonincreasing, calling the resulting vector $b$. Since this operation strictly increases the number of coefficients equal to $\frac{1}{\sqrt{n}}$, after finitely many operations, we arrive at the diagonal vector $(\frac{1}{\sqrt{n}}, \dots, \frac{1}{\sqrt{n}})$. Plainly, this operation preserves the $\ell_2$ norm. This allows to apply Lemma \ref{lm:diag-local} (conditioning on the value of $v = \sum_{j=2}^{n-1} a_j\xi_j$) which yields the following bound on the deficit coming from one operation transforming $a$ to~$b$,
\begin{align*}
&\E\left|a_1'\xi_1 + a_n'\xi_2 + v\right|^p - \E\left|a_1\xi_1 + a_n\xi_2 + v\right|^p \\
&\qquad\qquad\geq \frac{1}{d}(a_1^4+a_n^4 - a_1'^4-a_n'^4)\cdot \begin{cases}\beta_{p,d}(|v|^2+\sigma^2)^{\frac{p-4}{2}}, & 2 \leq p \leq 4, \\ \frac{p(p-2)}{8}\E|v+\sigma\xi|^{p-4}, & p > 4, \end{cases}
\end{align*}
where, as in the lemma, $\sigma^2 = a_1^2 + a_n^2 = a_1'^2 + a_n'^2$.
Moreover, averaging over the distribution of $v$ gives
\begin{itemize}
\item when $2 \leq p \leq 4$, by Jensen's inequality,
\[ 
\E(|v|^2+\sigma^2)^{\frac{p-4}{2}} \geq \left(\E|v|^2 + \sigma^2\right)^{\frac{p-4}{2}} =  \left(\sum_{j=1}^n a_j^2\right)^{\frac{p-4}{2}} = 1,
 \]

\item  when $p > 4$, by Lemma \ref{lm:Khin} (Khinchin's inequality),
\[ 
\E|v+\sigma\xi|^{p-4} \geq c_{\mathrm{Kh}}\left(\sum_{j=1}^n a_j^2\right)^{p-4} = c_{\mathrm{Kh}}.
 \]
\end{itemize}
Notably, $a_1^4+a_n^4 - a_1'^4-a_n'^4 = \|a\|_4^4 - \|b\|_4^4$, thus adding the contributions over all operations, these $\ell_4$ deficit differences add over a telescoping-type sum, leading to the bound
\begin{align*}
\E\left|\sum_{j=1}^n \frac{1}{\sqrt{n}}\xi_j\right|^p - \E\left|\sum_{j=1}^n a_j\xi_j\right|^p \geq  \tilde{c}_{p,d}\left( \|a\|_4^4- \left\|\left(\frac{1}{\sqrt{n}}, \dots, \frac{1}{\sqrt{n}}\right)\right\|_4^4  \right) =  \tilde{c}_{p,d}\delta(a),
\end{align*}
with
\[ 
  \tilde{c}_{p,d} = \frac{1}{d}\begin{cases}\beta_{p,d}, & 2 \leq p \leq 4, \\ c_{\mathrm{Kh}}\frac{p(p-2)}{8}, & p > 4, \end{cases}
 \]
which after plugging in the value of $\beta_{p,d}$ from \eqref{eq:beta-def} finishes the proof.
\end{proof}


\end{document}